\documentclass[11pt,reqno]{amsart}
\usepackage{cmap,mathtools}
\usepackage[T1]{fontenc}
\usepackage[utf8]{inputenc}
\usepackage[english]{babel}
\usepackage{amsfonts,amssymb,amsmath}
\usepackage{footnote,bigints}
\usepackage{enumitem}

\allowdisplaybreaks
\usepackage[pagewise, mathlines]{lineno}
\usepackage{graphicx}
\usepackage{epstopdf}
\usepackage{a4wide}
\usepackage{xcolor}
\usepackage[colorlinks=true,linktocpage,pdfpagelabels,
bookmarksnumbered,bookmarksopen]{hyperref}
\definecolor{ForestGreen}{rgb}{0.1,0.6,0.05}
\definecolor{EgyptBlue}{rgb}{0.063,0.1,0.6}
\hypersetup{
	colorlinks=true,
	linkcolor=EgyptBlue,         
	citecolor=ForestGreen,
	urlcolor=olive
}

\usepackage[hyperpageref]{backref}

\newtheorem{theorem}{Theorem}
\newtheorem{prop}[theorem]{Proposition}
\newtheorem{lem}[theorem]{Lemma}

\theoremstyle{definition}
\newtheorem{defn}[theorem]{Definition}
\newtheorem{rmk}[theorem]{Remark}

\let\OLDthebibliography\thebibliography
\renewcommand\thebibliography[1]{
	\OLDthebibliography{#1}
	\setlength{\parskip}{1pt}
	\setlength{\itemsep}{1pt plus 0.3ex}
}

\numberwithin{equation}{section}
\numberwithin{theorem}{section}
\numberwithin{equation}{section}
\numberwithin{theorem}{section}

\DeclarePairedDelimiter\norm{\lVert}{\rVert}

\usepackage{ulem}
\usepackage{xcolor}
\usepackage[colorlinks=true,linktocpage,pdfpagelabels,
bookmarksnumbered,bookmarksopen]{hyperref}
\definecolor{ForestGreen}{rgb}{0.1,0.6,0.05}
\definecolor{EgyptBlue}{rgb}{0.063,0.1,0.6}
\hypersetup{
	colorlinks=true,
	linkcolor=EgyptBlue,         
	citecolor=ForestGreen,
	urlcolor=olive
}

\usepackage[hyperpageref]{backref}
\usepackage[foot]{amsaddr}

\title [Szeg\"{o}-Weinberger type inequalities in space forms]{Szeg\"{o}-Weinberger type inequalities for symmetric domains in simply connected space forms}

\author{T. V. Anoop$^1$ \and Sheela Verma$^*$}
\address{$1$ Department of Mathematics, Indian Institute of Technology Madras, Chennai 36, India}
\email{anoop@iitm.ac.in}

\address{$*$  Corresponding author, Department of Mathematical Sciences, Indian Institute of Technology (BHU), Varanasi, India}
\email{sheela.mat@iitbhu.ac.in}

\begin{document}
	
	\begin{abstract}
		We consider the Neumann eigenvalue problem for Laplacian on a bounded multi-connected domain contained in simply connected space forms. Under certain symmetry assumptions on the domain, we prove Szeg\"{o}-Weinberger type inequalities for the first $n$ positive Neumann eigenvalues.
		
	\end{abstract}

	\maketitle
	
	\section{Introduction}
	
	Let $\Omega \subset {M}$ be a bounded domain with smooth boundary $\partial \Omega$ in a complete connected Riemannian manifold $M$ of dimension $n \geq 2$. Consider the Neumann eigenvalue problem on $\Omega$ 
	\begin{align}\label{Neumann problem}
		\begin{array}{rcll}
			\Delta \phi &=&\mu \phi & \text{ in } \Omega,\\
			\frac{\partial \phi}{\partial \nu} &=& 0 & \text{ on } \partial \Omega,
		\end{array}
	\end{align}
	where $\nu$ represents the outward unit normal to $\partial \Omega$. The set of all Neumann eigenvalues form a discrete sequence $0 = \mu_{1}(\Omega) < \mu_{2}(\Omega)  \leq \mu_3(\Omega) \leq \cdots \nearrow \infty$. This problem \eqref{Neumann problem} models the vibrations of a homogeneous free membrane.
	
	In 1952, Kornhauser and Stakgold \cite{KS} using a perturbation method, showed that the disk is a local maximizer of $\mu_{2}(\Omega)$ on the class of all simply connected planar domains of a given area. Later, Szeg\"{o} \cite{S} with the help of conformal mapping techniques (the `method of conformal transplantation'), proved that the ball is indeed a global maximum on the same class of domains. In 1956,  Weinberger \cite{W} generalized this result for the class of bounded Lipschitz domains in $\mathbb{R}^{n}$ with given volume, without the simply connectedness assumption. Let $\Omega^*$ be the ball centred at the origin in $\mathbb{R}^n$ of the same volume as $\Omega$. Then Weinberger's result can be stated as the following inequality (Szeg\"{o}-Weinberger inequality):
	\begin{equation}\label{ineq:SW}
		\mu_2(\Omega)\le \mu_2(\Omega^*).
	\end{equation}
	Later, Szeg\"{o} and Weinberger observed that (see page 634 of \cite{W}) Szeg\"{o}'s proof for \eqref{ineq:SW} for simply connected planar domains can be extended to the following inequality:
	\begin{equation}\label{ineq:SW2}
		\frac{1}{\mu_2(\Omega)}+\frac{1}{\mu_3(\Omega)}\ge \frac{1}{\mu_2(\Omega^*)}+\frac{1}{\mu_3(\Omega^*)}.
	\end{equation}
	Since $\mu_2(\Omega^*)=\mu_3(\Omega^*),$ and $\mu_3(\Omega)\ge \mu_2(\Omega),$ the above inequality clearly yields  \eqref{ineq:SW}.
	
	In \cite{B1}, Bandle  extended the \eqref{ineq:SW2} (and hence  \eqref{ineq:SW}) to  the case of an inhomogeneous free membrane using conformal mapping techniques, see also \cite[Theorem 3.12]{B2}. 
	In \cite[page 94 ]{C2}, Chavel remarked that using Weinberger's approach, one can easily prove the  Szeg\"{o}-Weinberger inequality for domains in space forms of constant negative sectional curvature. However,  in the case of space forms of constant positive sectional curvature, Weinberger's approach will work only for the domains that are contained in sufficiently small geodesic balls. For example, for dimension 2, $\Omega$ must  contained in a geodesic ball of radius $\pi/(4\sqrt{\kappa}), $ where $\kappa$ is the constant positive sectional curvature of the space form, see  \cite[page 80]{C1}). In \cite{AB},   Ashbaugh and Benguria improved this result for the domains contained in a geodesic ball of radius $\pi/(2\sqrt{\kappa})$ for all dimensions.  
	For extension of \eqref{ineq:SW} to domains in more general manifolds, see \cite{ARGS, SV, KW, Xu}.
	It is known that
	\begin{equation}\label{ineq:SW-k}
		\mu_3(\Omega)< 2^\frac{2}{n} \mu_3(\Omega^*)
	\end{equation}
	due to  Girouard,  Nadirashvili, \&  Polterovich for bounded Jordan domains in $\mathbb{R}^2$
	\cite{GNP}  and due to  Bucur \& Henrot \cite{BH} for general domains in  $\mathbb{R}^n$ with $n\geq 2.$ However, to 
	establish   the Szeg\"{o}-Weinberger type inequality for the higher Neumann eigenvalues one needs to make certain symmetry restriction  on $\Omega$.

	For $k\in \mathbb{N},$ and $i,j\in \left\{1,2,\ldots, n\right\}$, let $R_{i,j}^{\frac{2 \pi}{k}}$ be  the rotation (in the anti-clockwise direction with respect to the origin) by an angle $\frac{2 \pi}{k}$ in the coordinate plane $(x_{i}, x_{j}).$ A domain $\Omega \subset \mathbb{R}^{n}$ is said to be \textit{symmetric of order $k$} with respect to the origin, if there exists a rotation $R$ on $\mathbb{R}^n$ such that  $R_{i,j}^{\frac{2 \pi}{k}}(R(\Omega)) = R(\Omega), \forall\, i,j \in \left\{1,2,\ldots, n\right\}.$ A domain $\Omega \subset \mathbb{R}^{n}$ is said to be \textit{centrally symmetric} with respect to the origin, if $x \in \Omega$ if and only if $-x \in \Omega$.
	
	In \cite[Section 5]{H}, Hersch proved that for any Jordan domain $\Omega \subset \mathbb{R}^{2}$, symmetric of order $k \geq 3$, 
	\begin{equation}\label{ineq:Hersch}
		\mu_{3}(\Omega) \leq \mu_{3}(\Omega^*).
	\end{equation}  
	For a smooth simply connected domain $\Omega$ with such symmetries, Ashbaugh and Benguria \cite[Lemma 4.1]{AB1} proved that,  $\mu_2(\Omega)=\mu_3(\Omega).$ In particular, this yields  \eqref{ineq:Hersch}  from \eqref{ineq:SW}. If $\Omega$ is  symmetric of order $4$, then they  obtained \eqref{ineq:Hersch}, without the simply connectedness assumption on $\Omega$, see \cite[Theorem 4.3]{AB1}.  However, for a Jordan domain $\Omega$ with symmetry of order $4$, Hersch \cite[Section 5]{H}  obtained  \begin{equation}\label{ineq:Hersch2}
		\mu_{4}(\Omega) \leq \mu_{4}(\Omega^*).
	\end{equation}
	For further inequalities involving $\mu_{k}(\Omega)$ for domains with symmetry of order $k \geq 2$, see \cite{EP1,EP2}.
	
	Recently, the inequalities for higher eigenvalues  have been extended for multiply connected domains by considering the eigenvalues of an appropriate concentric annular region instead of  the  ball $\Omega^*.$ In \cite{ABD}, authors compared the Neumann eigenvalues of Laplacian on a Lipschitz domain $\Omega = \Omega_{out} \setminus \overline{\Omega}_{in}$  and a concentric annular region $B_{\beta} \setminus \overline{B}_{\alpha}$, where
	\begin{enumerate}[label = (\roman*)]
		\item  $\Omega_{in}$ is compactly contained in  $\Omega_{out},$
		\item $0 \leq \alpha< \beta$ is such that $B_{\alpha} \subset \Omega_{in}$ and $|\Omega| = |B_{\beta} \setminus \overline{B}_{\alpha}|$.
	\end{enumerate}
	They proved that, if $\Omega$ is either symmetric of order 2 or centrally symmetric, then $$\mu_2({\Omega})\leq \mu_2({B_{\beta} \setminus \overline{B}_{\alpha}}),$$ and if domain $\Omega$ is symmetric of order 4, then 
	\begin{equation}\label{ineq:4sym}
		\mu_i({\Omega})\leq \mu_i({B_{\beta} \setminus \overline{B}_{\alpha}}) \text{ for } i=2,3,\ldots, n+2.
	\end{equation}  
	In \cite{SVGS}, a similar result has been proved for the first nonzero Neumann eigenvalue for centrally symmetric doubly connected domains contained in noncompact rank-1 symmetric spaces with the additional assumption that inner domain $\Omega_{in}$ is a geodesic ball.
	
	In this article, we partially extend \eqref{ineq:4sym} for the domains with holes contained in a simply connected space form. It is well known that a smooth, connected, and simply-connected complete Riemannian manifold with constant curvature is isometric to either of $\mathbb{R}^{n}$(for zero curvature), $\mathbb{S}^{n}$(for positive curvature) and $\mathbb{H}^{n}$(for negative curvature). Due to this fact, we prove our main result for $M= \mathbb{S}^{n}$ and $\mathbb{H}^{n}$.
	
	Now we provide the notations used in this article and state the notion of $k$-symmetry for the domains in space forms.
	
	\noindent \textbf{Notations:} Throughout this article, $M$ represents either $\mathbb{S}^{n}$ or $\mathbb{H}^{n}$, $n \geq 2$. Let $p= (1,0,0, \ldots, 0)\in M$ be a fixed point and  $B_{r}$  denotes the geodesic ball in $M$ of radius $r$.  Let $\exp_{p} : T_{p}(M) \rightarrow M$ be the exponential map and $X = (X_1, X_2, \ldots, X_n)$ be the geodesic normal coordinates centered at $p$. We identify any domain $\Omega$ in $M$ with $\exp_{p}^{-1}(\Omega)$. 
	
	\begin{defn}
		Let a domain $\Omega \subset M$ be represented as $\Omega = \exp_{p}(U)$ for $p \in M$ and $U \subset T_{p}(M)$. Then  $\Omega$ is said to be \textit{symmetric of order $k$} with respect to $p,$ if $U$ is symmetric of order $k$ with respect to the origin. The domain $\Omega$ is said to be \textit{geodesically symmetric} or \textit{centrally symmetric} with respect to the point $p$, if $U$ is centrally symmetric about the origin. 
	\end{defn}
	
	\begin{rmk}
		For $n=2$, a domain $\Omega$ in $M$ is centrally symmetric if and only if it is symmetric of order $2$. For $n \geq 3$, the central symmetry does not imply the symmetry of order $2$ \cite[Lemma 5.2]{ABD}. If $\Omega$ is a domain having symmetry of order $k \in \mathbb{N}$ with $k \neq 2,4$, then $\Omega$ is either a geodesic ball or a concentric annular domain. For a proof, see \cite[Proposition 5.1]{ABD}.
	\end{rmk}
	
	We make the following assumptions on the domain $\Omega:$
	
	\begin{enumerate}[label={$\mathbf{(A_1)}$}]
		\item \label{assumption1} Let  $\Omega = \Omega_{out} \setminus \overline{\Omega}_{in}$ be a domain in $M$, where $\overline{\Omega}_{in} \subset \Omega_{out}$. Here domain $\Omega$ might possess other holes except $\Omega_{in}$ or might possess no holes at all i.e., $\Omega_{in} = \emptyset$. If $\Omega_{in}$ is non empty, then we assume that $p \in \Omega_{in}$. In the case of $M= \mathbb{S}^{n}$, we additionally assume that  $\Omega$ is contained in a geodesic ball of radius $\pi/2$ with center at $p.$ 
	\end{enumerate}
	
	\begin{enumerate}[label={$\mathbf{(A_2)}$}]
		\item\label{assumption2} Let $B_{R_{1}}$ and $B_{R_{2}}$ be the geodesic balls in $M$, centered at $p$, of radius $R_{1}$ and $R_{2}$, respectively such that $B_{R_{1}} \subset \Omega_{in}$ and Vol$(\Omega)$ = Vol$(B_{R_{2}} \setminus \overline{B}_{R_{1}})$. 
	\end{enumerate} 
	
	\begin{rmk}
		It follows from assumption \ref{assumption1} that for our choice of $\Omega$, there exists a neighborhood $U$ of the origin in $T_{p}M$ such that $\exp_{p}: U \rightarrow \Omega$ is a diffeomorphism.
	\end{rmk}
	
	The main result of this article is as follows.
	\begin{theorem} \label{main result}
		Let  $\Omega \subset M,$ $R_1$ and $R_2$ be as given in \ref{assumption1} and \ref{assumption2}.  In addition,
		\begin{enumerate}
			\item if $\Omega$ is either symmetric of order $2$ or centrally symmetric with respect to the point $p$, then 
			\begin{align} \label{main result: order 2}
				\mu_2({\Omega})\leq \mu_2(B_{R_{2}} \setminus \overline{B}_{R_{1}}),
			\end{align}
			\item if $\Omega$ is symmetric of order $4$ with respect to the point $p$, then for $i=2,\ldots, n+1$,
			\begin{align}\label{main result: order 4}
				\mu_i({\Omega})\leq \mu_i(B_{R_{2}} \setminus \overline{B}_{R_{1}}) = \mu_2(B_{R_{2}} \setminus \overline{B}_{R_{1}}).
			\end{align}
		\end{enumerate}
	\end{theorem}
	
	\begin{rmk}
		If $\Omega_{in}=\emptyset$, then \eqref{main result: order 4} extend \cite[Theorem 4.3]{AB1} to $\mathbb{S}^{n}$ and $\mathbb{H}^{n}$, $n \geq 2$.  To the best of our knowledge, this is the first result that gives the Szeg\"{o}-Weinberger type inequality for the higher Neumann eigenvalues on space forms. In \cite{FL}, authors proved that the union of two disjoint geodesic balls of the same volume maximise the third Neumann eigenvalue among the regions of given
		volume in hyperbolic spaces.
	\end{rmk}
	
	\begin{rmk}
		As an easy consequence of the above theorem, we also obtain an isoperimetric inequality for the harmonic mean of the first $n$ nonzero Neumann eigenvalues of  Lapacian on a domain  $\Omega$ which is symmetric of order $4$ as given below:
		\begin{align*}
			\frac{1}{\mu_2({\Omega})} + \frac{1}{\mu_3({\Omega})} + \cdots + \frac{1}{\mu_{n+1}({\Omega})} \geq \frac{n}{\mu_2(B_{R_{2}} \setminus \overline{B}_{R_{1}})}.
		\end{align*}
		Whereas in \cite{BBC, WX}, authors proved a similar bound for the harmonic mean of the first $(n-1)$ nonzero Neumann eigenvalues in terms of the first nonzero Neumann eigenvalue of the geodesic ball, without any symmetry assumption. 
	\end{rmk}
	
	The remaining part of this article is organised as follows. Section \ref{Neumann problem on annular domain} is devoted to the study of Neumann eigenvalues and eigenfunctions on an annular domain contained in $M$. This section also includes some relations satisfied by these eigenvalues and eigenfunctions, which are important to prove our main result. We give detailed computation of geodesic normal coordinates on $M ( = \mathbb{S}^{n}$ and $\mathbb{H}^{n})$ and derive some of their properties in Section \ref{Sec:geodesic normal coordinates}. In this section, we also obtain some integral identities, which helps us to find test functions for the variational characterisation of Neumann eigenvalues In Section \ref{sec:main result}, we prove the main result and provide some concluding remarks.
	
	\section{The Neumann problem on an annular domain} \label{Neumann problem on annular domain}
	This section describes the Neumann eigenvalues on an annular domain as the eigenvalues of certain Sturm–Liouville eigenvalue problems. These Sturm–Liouville problems are determined by the eigenvalues of $\Delta_{\mathbb{S}^{n-1}}$.  We also prove some inequalities between the eigenvalues of distinct Sturm–Liouville problems which are essential for proving the main result.
	
	\subsection{The spherical harmonics}
	A spherical harmonic $Y$ on $\mathbb{S}^{n-1}$ of degree $k \geq 0$ is the restriction of  $\Tilde{Y}$, a harmonic homogeneous polynomial of degree $k$ on $\mathbb{R}^{n}$, to $\mathbb{S}^{n-1}$. Let $\mathcal{H}_{k}, k \geq 0$ represent the space of harmonic homogeneous polynomials of degree $k$ on $\mathbb{R}^{n}$. Then dim $\mathcal{H}_{k} = \binom{k+n-1}{n-1} - \binom{k+n-3}{n-1}$. Notice that,
	\begin{align*}
		& \mathcal{H}_{0} = \text{ span } \{1\} \\
		& \mathcal{H}_{1} = \text{ span } \{x_{i} : i \in \{1,\ldots, n\} \\
		& \mathcal{H}_{2} = \text{ span } \{x_{i} x_{j}, x_{1}^{2} - x_{k}^{2} : i,j \in \{1,\ldots, n\} \text{ and } k \in \{2,\ldots, n\}\}.
	\end{align*}
	The following proposition describe the set of eigenvalues and the eigenfunctions of $\Delta_{\mathbb{S}^{n-1}}$ \cite[Sections 22.3, 22.4]{SH}. 
	\begin{prop} \label{multiplicity of eigenvalue}
		The set of all eigenvalues of $\Delta_{\mathbb{S}^{n-1}}$ is $\{k(k+n-2) : k \in \mathbb{N} \cup \{0\} \}$. The eigenfunctions corresponding to each eigenvalue $k(k+n-2)$ are the spherical harmonics of degree $k$ and thus the multiplicity of $k(k+n-2)$ is equal to dim $\mathcal{H}_{k}$.
	\end{prop}

	\subsection{The Neumann eigenvalues and eigenfunctions on annular domain}
	Fix $p \in M$, we describe the Neumann eigenvalues and corresponding eigenfunctions on annular domain $(B_{R_{2}} \setminus \overline{B}_{R_{1}}) \subset M$ centered at the point $p$. 
	
	Recall that, the Riemannian metric $g_{M}$ on $M$ in terms of geodesic polar coordinates is of the form $g_{M}(r,\Theta)=dr^2+\sin_{M}^2(r)g_{0}(\Theta)$. Here $(r,\Theta)\in [0,L] \times \mathbb{S}^{n-1}$, $g_0$ is the canonical metric on the $(n-1)$-dimensional unit sphere $\mathbb{S}^{n-1}$ and function $\sin_{M}(r)$ is defined as
	\begin{align*}
		\sin_{M}(r) :=
		\begin{cases}
			\sin r, & M= \mathbb{S}^{n}, \\
			\sinh r, & M= \mathbb{H}^{n}.
		\end{cases}
	\end{align*}
	
	Let $S(r)$ be the geodesic sphere of radius $r$ centered at $p$ and tr$(A(r))$ be the trace of the second fundamental form $A(r).$ Then  $\Delta$ on $M$ can be decomposed in terms of $\Delta_{S(r)}$ as  given below:
	\begin{align*}
		\Delta = - \frac{d^{2}}{dr^{2}} - \text{ tr}(A(r)) \frac{d}{dr} + \Delta_{S(r)}.
	\end{align*}
	For $M= \mathbb{S}^{n}$ and $\mathbb{H}^{n}$, one can verify that 
	\begin{align*}
		\text{ tr(A(r))} = \frac{1}{\sin_{M}^{n-1}(r)} \frac{d}{dr} \left(\sin_{M}^{n-1}(r) \right) \text{ and } \Delta_{S(r)} = \frac{1}{\sin_{M}^{2}(r)} \Delta_{\mathbb{S}^{n-1}}.
	\end{align*}
	Thus for a function $f(r,\Theta) = u(r)v(\Theta)$ defined on $(B_{R_{2}} \setminus B_{R_{1}})$,
	\begin{align*}
		\Delta (u(r)v(\Theta))&= \left( - \frac{d^{2}u}{dr^{2}} - \frac{1}{\sin_{M}^{n-1}(r)} \frac{d}{dr} \left(\sin_{M}^{n-1}(r) \right) \frac{du}{dr} \right) v(\Theta) + \Delta_{S(r)} (u(r)v(\Theta)) \\
		&= -\frac{1}{\sin_{M}^{n-1}(r)} \frac{d}{dr} \left(\sin_{M}^{n-1}(r) \frac{d}{dr}u \right) v(\Theta) +\frac{u}{\sin_{M}^2(r)}\Delta_{\mathbb{S}^{n-1}} v(\Theta).
	\end{align*}
	If $v$ is spherical harmonic of degree $k$, then we obtain
	\begin{align*}
		\Delta (u(r)v(\Theta)) = -\frac{1}{\sin_{M}^{n-1}(r)} \frac{d}{dr} \left(\sin_{M}^{n-1}(r) \frac{d}{dr}u \right)v(\Theta)+\frac{u}{\sin_{M}^2 (r)} k(k+n-2) v(\Theta),
	\end{align*}
	Thus by separation of variable technique, for an eigenpair $(\mu, f)$ of the Neumann problem on $B_{R_{2}} \setminus B_{R_{1}}$, $f$ is of the form  $f(r,\Theta) = u(r)v(\Theta)$, where $v$ is a spherical harmonic corresponding to the eigenvalue $k(k+n-2)$ and $u$ satisfies the following Sturm–Liouville eigenvalue problem:
	\begin{align} \label{eqn:separationofvariable}
		\begin{array} {rcl}
			- u''(r) - \frac{(n-1) \sin_{M}'(r)}{\sin_{M}(r)} u'(r) + \frac{k(k+n-2)}{\sin_{M}^{2}(r)} u(r) = \mu u(r)
		\end{array}
	\end{align}
	with the boundary conditions:
	\begin{equation}\label{BCs}
		u'(R_{1})=0, \quad u'(R_{2})=0.
	\end{equation}
	
	Notice that, for each $k\in \mathbb{N} \cup \{0\},$ the eigenvalues of the above Sturm–Liouville problem form an increasing sequence $0 \leq \mu_{k,1} < \mu_{k,2} < \mu_{k,3} < \cdots \nearrow \infty$ and each eigenvalue $\mu_{k,j}$ has multiplicity one and the corresponding eigenfunction vanishes exactly $(j-1)$ times in the interval $(R_{1}, R_{2})$. For  $k\in \mathbb{N} \cup \{0\},$ and $u\in H^{1}((R_{1}, R_{2}); \sin_{M}^{n-1}(r))\setminus\{0\}$, consider the Rayleigh quotient
	\begin{equation}
		\mathcal{R}_k(u):= \frac{\int_{R_{1}}^{R_{2}}\left( \left( u'(r) \right)^{2} + \frac{k(k+n-2)}{\sin_{M}^{2}(r)} u^{2}(r) \right)\sin_{M}^{n-1}(r) \ dr}{\int_{R_{1}}^{R_{2}} u^{2}(r) \sin_{M}^{n-1}(r) \ dr }. 
	\end{equation}
	Now, a variational characterization of  $\mu_{k,j}$ can be  given as below:
	\begin{align*}
		\mu_{k,j} = \min_{E \in \mathcal{H}_{j}} \max_{u \in E\setminus \{0\}} \mathcal{R}_k(u),
	\end{align*}
	where $\mathcal{H}_{j}$  is the set of all $j$-dimensional subspaces  of $H^{1}\left((R_{1}, R_{2}); \sin_{M}^{n-1}(r)\right)$. 
	Since $\mathcal{R}_{k+1}(u)>\mathcal{R}_k(u)$, for each $j\in \mathbb{N}$, we also observe that 
	\begin{equation}\label{j-increase}
		\mu_{0,j} < \mu_{1,j} < \mu_{2,j} < \cdots \nearrow \infty.
	\end{equation}
	
	\begin{rmk}\label{Dirichlet evs}
		For each $k\in \mathbb{N}\cup\{0\}$, Sturm–Liouville problem \eqref{eqn:separationofvariable} with the boundary conditions  $u(R_1)=u(R_2)=0$  the set of  eigenvalues form  a sequence of the form
		$0 <\lambda_{k,1} < \lambda_{k,2} < \lambda_{k,3} < \cdots \nearrow \infty$.  These eigenvalues have the following variational characterization:
		\begin{align*}
			\lambda_{k,j} = \min_{E \in \mathcal{H}_{j}} \max_{u \in E\setminus \{0\}} \mathcal{R}_k(u) ,
		\end{align*}
		where $\mathcal{X}_{j}$ is the set of all $j$-dimensional subspaces of the Sobolev space $H^{1}_0((R_{1}, R_{2}); \sin_{M}^{n-1}(r))$. For each $k\in \mathbb{N}\cup\{0\}$ and $j\in \mathbb{N},$ from the variational characterizations, we clearly have $\mu_{k,j}\le \lambda_{k,j}.$ Moreover, using the simplicity of eigenvalues, one can show that 
		\begin{equation}\label{compare-evs}
			\mu_{k,j}< \lambda_{k,j}, \forall\, k\in \mathbb{N}\cup\{0\}, \forall\, j\in \mathbb{N}.
		\end{equation}
	\end{rmk}
	
	\begin{rmk}
		If $\lambda_{i}(S(r))$ denote the $i^{th}$ eigenvalue of $\Delta_{S(r)}$, then $\lambda_{0}(S(r)) = 0$ and $\lambda_{1}(S(r)) = \frac{n-1}{\sin_{M}^{2}(r)}$. Also observe that $(\text{tr}(A(r)))' = - \lambda_{1}(S(r))$.
	\end{rmk}
	The observations in the above remarks lead to the following lemma:
	\begin{lem}
		\label{eigenvalue comparison} For $j\in \mathbb{N},$
		$\mu_{0,j+1} = \lambda_{1,j}$ and hence $\mu_{1,j}<\mu_{0,j+1}.$
	\end{lem} 
	
	\begin{proof}
		Let $f$ be an eigenfunction corresponding to the eigenvalue $\mu_{0,j+1}$ for some $j\in \mathbb{N}.$ Thus $f$ satisfy the following Sturm–Liouville problem:
		\begin{align*} \label{f equation}
			- \frac{d^{2}f}{dr^{2}} - \text{ tr}(A(r)) \frac{df}{dr} & =\mu_{0,j+1} f, \\
			f'(R_1)=0 \, & f'(R_2)=0.
		\end{align*}
		By differentiating the above equation and using fact that $(\text{tr}(A(r)))' = - \lambda_{1}(S(r))$, we obtain,
		\begin{align*}
			- \frac{d^{2}f'(r)}{dr^{2}} - \text{ tr}(A(r)) \frac{df'(r)}{dr} + \lambda_{1}(S(r)) f'(r) = \mu_{0,j+1} f'(r).
		\end{align*}
		Therefore, $u=f'$ is an eigenfunction corresponding to the eigenvalue $\mu_{0,j+1}$ of  \eqref{eqn:separationofvariable} with  $k=1$ and satisfies the  boundary conditions $u(R_1)=u(R_2)=0$. Thus $\mu_{0,j+1}\in \{\lambda_{1,i}: i\in \mathbb{N} \}.$
		From \eqref{j-increase} and \eqref{compare-evs}, also have
		$$\mu_{0,j+1}<\mu_{1,j+1}<\lambda_{1,j+1}.$$
		Thus, we must have $\mu_{0,2}=\lambda_{1,1}$ and subsequently  $\mu_{0,i+1} = \lambda_{1,i}$ for $i=2,3,\ldots,j-1.$
		
	\end{proof}
	The following lemma plays an integral part in proving our main theorem. For an analogue of this lemma for the Euclidean domains, see \cite[Lemma 2.7]{ABD}.
	\begin{lem} \label{lem: increasing eigenfunction}
		For $k\in \mathbb{N},$ let $u(r)$ be the positive eigenfunction corresponding to the eigenvalue $\mu_{k,1}$ of the Sturm–Liouville problem \eqref{eqn:separationofvariable} with the boundary conditions \eqref{BCs}. Then 
		\begin{enumerate}[label = (\roman*)]
			\item $\mu_{k,1} = \frac{k(k+n-2)}{\sin_{M}^{2}(b)} \text{ for some } b \in (R_{1}, R_{2}),$
			\item $u$ is strictly  increasing  on $(R_{1}, R_{2}),$
			\item for each $r\in (R_{1}, R_{2}),$  
			\begin{equation} \label{comparison ineq.}
				\left( \frac{k(k+n-2)}{\sin_{M}^{2}(r)} - \mu_{k,1} \right) u^{2}(r) \geq \left( \frac{k(k+n-2)}{\sin_{M}^{2}(R_{2})} - \mu_{k,1} \right) u^{2}(R_{2}).
			\end{equation}
		\end{enumerate}
		
	\end{lem}
	\begin{proof} 
		\noindent $(i)\,$ Notice that  $u$ satisfies
		\begin{align} \label{eqn:SLmuk}
			\frac{d}{dr} \left(\sin_{M}^{n-1}(r) \frac{d}{dr}u(r) \right) &= \left( \frac{k(k+n-2)}{\sin_{M}^{2}(r)} - \mu_{k,1} \right) u(r) \sin_{M}^{n-1}(r), r\in (R_1,R_2),\\
			u'(R_{1})=0, \; & u'(R_{2})=0. \nonumber
		\end{align}
		Let $\Phi(r)= \sin_{M}^{n-1}(r) u'(r)$. Then $\Phi(R_1) = 0 = \Phi(R_2)$.  Thus there exists $b \in (R_{1}, R_{2})$ such that $\Phi'(b)= 0$ and hence \eqref{eqn:SLmuk} yields $\mu_{k,1} = \frac{k(k+n-2)}{\sin_{M}^{2}(b)}$.
		
		\vspace{5mm}
		\noindent $(ii)\,$ Since $\sin_{M}^{n-1}(r) u(r)>0,$ from \eqref{eqn:SLmuk} we obtain
		\begin{align*}
			\Phi '(r)> 0 \text{ for } r \in (R_1,b), \text{ and } 
			&\Phi '(r)< 0 \text{ for } r \in (b,R_2),
		\end{align*}
		where  $b$ as given in $(i).$
		Now, as $\Phi(R_1) = 0 = \Phi(R_2)$, we easily conclude that  $\Phi(r)>0$ for every $r\in (R_1,R_2).$ Thus $u'(r)>0$ for  $r\in (R_1,R_2)$ as required. 
		
		\vspace{5mm}
		\noindent $(iii)\,$ For $r\in (R_1,b],$  \eqref{comparison ineq.} holds trivially. For  $r \in (b, R_{2}),$  using $(i),$ we get 
		\begin{align*}
			0> \frac{k(k+n-2)}{\sin_{M}^{2}(r)} - \mu_{k,1} >\frac{k(k+n-2)}{\sin_{M}^{2}(R_{2})} - \mu_{k,1}.
		\end{align*}
		Since $u$ is strictly increasing and positive, the above inequality yields  \eqref{comparison ineq.} for $r \in (b, R_{2}).$
	\end{proof}

	\begin{rmk} \label{rmk: multiplicity of annular eigenvalues}
		The set of Neumann eigenvalues of $-\Delta$ on $(B_{R_{2}}\setminus \overline{B}_{R_{1}})$ is given by
		\begin{align*}
			\{ \mu_{i} (B_{R_{2}} \setminus \overline{B}_{R_{1}}) \} _{i \in \mathbb{N}} = \{ \mu_{k,j} \} _{k \in \mathbb{N} \cup \{ 0 \} ,j \in \mathbb{N}}.
		\end{align*}
		Therefore, $\mu_1(B_{R_{2}} \setminus \overline{B}_{R_{1}})=\mu_{0,1}=0$ and by  Lemma \ref{eigenvalue comparison} and Proposition \ref{multiplicity of eigenvalue}, we conclude that 
		$$\mu_2(B_{R_{2}} \setminus \overline{B}_{R_{1}}) = \mu_3(B_{R_{2}} \setminus \overline{B}_{R_{1}}) = \cdots = \mu_{n+1}(B_{R_{2}} \setminus \overline{B}_{R_{1}}) = \min\{\mu_{0,2},\mu_{1,1}\}=\mu_{1,1}.$$ The corresponding eigenfunctions are $u(r) \frac{X_i}{r}$, $1 \leq i \leq n$, where $(X_1, X_2,\ldots, X_n)$ is a geodesic normal coordinates centered at the point $p$ and $u(r)$ is  an eigenfunction  corresponding to the eigenvalue  $\mu_{1,1}$ of   \eqref{eqn:separationofvariable} with boundary conditions \eqref{BCs}.
	\end{rmk}
	
	A proof for the next proposition can be given using the above lemma and similar arguments as given in the proof of Proposition 2.8 of \cite{ABD}. 
	\begin{prop} \label{prop:comparison with eigenvalue}
		Let $\Omega, R_1$ and $R_2$ be as given in \ref{assumption1} and \ref{assumption2}. For $k \in \mathbb{N},$ let $u_k$ be a positive eigenfunction of \eqref{eqn:separationofvariable} and \eqref{BCs} corresponding to the eigenvalue $\mu_{k,1}.$ Then
		\begin{align} \label{comparison with eigenvalue}
			\frac{\int_{\Omega} \left( (G'_{k}(r))^{2} + \frac{k(k+n-2)}{\sin_{M}^{2}(r)} G_{k}^{2}(r) \right) dV}{\int_{\Omega}G_{k}^{2}(r) \ dV} \leq \mu_{k,1},
		\end{align}
		where
		\begin{align*}
			G_{k}(r) = 
			\begin{cases}
				u_k(r), & \text{ if } r \in (R_{1}, R_{2}), \\
				u_k(R_{2}), & \text{ if } r \geq R_{2}.
			\end{cases}
		\end{align*}
		Furthermore, equality holds in \eqref{comparison with eigenvalue} if and only if $\Omega$ coincides with $B_{R_{2}} \setminus B_{R_{1}}$.
	\end{prop}
	\section{Geodesic normal coordinates on Sphere and hyperbolic space} \label{Sec:geodesic normal coordinates}
	
	In this section, we will construct the geodesic normal coordinates on $M$ that satisfy some properties similar to that of standard Euclidean coordinates.
	
	\subsection{ Geodesic normal coordinates on the unit $n$-sphere $\mathbb{S}^{n}$}
	We consider the following parametrization for $q=(q_1,q_2,\ldots, q_{n+1})\in \mathbb{S}^{n}:$ 
	\begin{align*}
		q_{1} &= \cos (\phi_1),\; q_{2} = \sin (\phi_1) \cos (\phi_2),\ldots, q_{n} = \sin (\phi_1) \sin (\phi_2) \sin (\phi_3) \cdots \cos (\phi_{n}) \\
		q_{n+1}&= \sin (\phi_1) \sin (\phi_2) \sin (\phi_3) \cdots \sin (\phi_{n}).
	\end{align*}
	with $ \phi_1, \phi_2,\ldots,\phi_{n-1} \in [0, \pi] $ and $\phi_{n} \in [0, 2 \pi] $. Then the tangent space $T_{q}\mathbb{S}^{n}$ at any point $q \in \mathbb{S}^{n}$ is spanned by $\frac{\partial}{\partial \phi_1} \big|_{q}, \frac{\partial}{\partial \phi_2}\big|_{q}, \ldots, \frac{\partial}{\partial \phi_n}\big|_{q}$, where \\
	\begin{align*}
		& \frac{\partial}{\partial \phi_1}\bigg|_{q} = \left( - \sin (\phi_1), \cos (\phi_1) \cos (\phi_2), \ldots, \cos (\phi_1) \sin (\phi_2) \sin (\phi_3) \cdots \sin (\phi_{n}) \right), \\
		& \frac{\partial}{\partial \phi_2}\bigg|_{q} = \left( 0, - \sin (\phi_1) \sin (\phi_2), \ldots, \sin (\phi_1) \cos (\phi_2) \sin (\phi_3) \cdots \sin (\phi_{n}) \right), \\
		& \hspace{4cm}\vdots \\
		& \frac{\partial}{\partial \phi_n}\bigg|_{q} = \left( 0, 0,  \ldots,0, - \sin (\phi_1) \sin (\phi_2) \sin (\phi_3) \cdots \sin (\phi_{n}), \sin (\phi_1) \sin (\phi_2) \sin (\phi_3) \cdots \cos (\phi_{n}) \right).
	\end{align*}
	This gives the Riemannian metric on $\mathbb{S}^{n}$, 
	\begin{align*}
		g_{\mathbb{S}^{n}} &= {d \phi_1}^{2} + \sin^{2} (\phi_1) {d \phi_2}^{2} +  \cdots + \sin^{2} (\phi_1) \sin^{2} (\phi_2) \cdots \sin^{2} (\phi_{n-1})  {d \phi_n}^{2}.
	\end{align*}
	For $\textbf{v}, \textbf{w} \in T_{p}\mathbb{S}^{n}$, we denote $g_{\mathbb{S}^{n}}(p) (\textbf{v},\textbf{w})$ and   $g_{\mathbb{S}^{n}}(p) (\textbf{v},\textbf{v})$   by $\langle \textbf{v}, \textbf{w} \rangle_{p}$ and $\|\textbf{v}\|_{p}$, respectively.
	
	Next we find a geodesic normal coordinates of a point $q \in \mathbb{S}^{n}$ centered at $p$.  Let $\{e_{i}\}_{i=1}^{n+1}$ be the standard orthonormal basis of $\mathbb{R}^{n+1}$. For any $q \in \mathbb{S}^{n}$, there exist $\textbf{v} \in  T_{p}\mathbb{S}^{n}$, $\|\textbf{v}\|_{p} = 1$ and $t \in \mathbb{R}^{+} \cup \{0\}$ such that
	\begin{align} \label{exp map on sphere}
		q = \exp_{p} t\textbf{v} = p \cos t + \textbf{v} \sin t. 
	\end{align}
	Now from the parametric representation of $q$, we easily get  $$t = \phi_1\;  \text{ and  }\textbf{v} = \left( 0, \cos (\phi_2), \sin (\phi_2) \cos (\phi_3),\ldots, \sin (\phi_2) \sin (\phi_3) \cdots \sin (\phi_n) \right).$$
	Therefore,
	with respect to the orthonormal  basis  $\left( e_{2}, e_{3}, \ldots, e_{n+1} \right)$ of $T_{p}\mathbb{S}^{n}$ at $p$ the geodesic normal coordinates
	of $q$   are given by 
	\begin{equation}\label{Geodesic normal coordinates on sphere}
		X_i(q)= t\langle \textbf{v}, e_{i+1} \rangle_{p}=\left\{\begin{array}{cc}
			\phi_1\cos (\phi_2) & i=1, \\
			\phi_1\sin (\phi_2) \cdots \sin (\phi_i) \cos (\phi_{i+1}) & \quad 2 \leq i \leq n-1,  \\
			\phi_1\sin (\phi_2) \cdots \sin (\phi_3) \sin (\phi_{n}) & i=n.
		\end{array}\right.
	\end{equation}
	
	


	\subsection{Geodesic normal coordinates on hyperbolic space}
	For $n$-dimensional hyperbolic space $\mathbb{H}^{n}$, consider the hyperboloid model $$\{(x_{1}, x_{2}, \ldots, x_{n+1}) \in \mathbb{R}^{n+1} | -x_{1}^{2} + x_{2}^{2} + \cdots + x_{n+1}^{2} = -1, x_{1} > 0\}$$ with the Riemannian metric $$ ds^{2} = -dx_{1}^{2} + dx_{2}^{2} + \cdots + dx_{n+1}^{2}.$$ For $q=(q_1,q_2,\ldots, q_{n+1})\in \mathbb{H}_n,$ consider the following parametrization: 
	\begin{align*}
		q_1&= \cosh (\phi_1),\; q_2=\sinh (\phi_1) \cos (\phi_2), \ldots, \\
		q_n &= \sinh (\phi_1) \sin (\phi_2) \sin (\phi_3) \cdots \cos (\phi_{n}), \;
		q_{n+1}= \sinh (\phi_1) \sin (\phi_2) \sin (\phi_3) \cdots \sin (\phi_{n}),
	\end{align*}
	where $\phi_1 \in [0, \infty)$,  $\phi_2,\ldots,\phi_{n-1} \in [0, \pi] $ and $\phi_{n} \in [0, 2 \pi].$
	Then its Riemannian metric will take form
	\begin{align*}
		ds^{2}   &= {d \phi_1}^{2} + \sinh^{2} (\phi_1) {d \phi_2}^{2} + \sinh^{2} (\phi_1) \sin^{2} (\phi_2) {d \phi_3}^{2} + \cdots + \sinh^{2} (\phi_1) \sin^{2} (\phi_2) \cdots \\
		& \quad \cdots \sin^{2} (\phi_{n-1})  {d \phi_n}^{2}.
	\end{align*}
	Now,  any  $q \in \mathbb{H}^{n}$ can be represented as
	\begin{align*}
		q = \exp_{p} t\textbf{v} = p \cosh t + \textbf{v} \sinh t,
	\end{align*}
	for some $\textbf{v} \in T_{p}\mathbb{H}^{n}$, $ds^{2}(\textbf{v}, \textbf{v})_{p} = 1$ and $t \in \mathbb{R}^{+} \cup \{0\}$.
	Thus, by the parametric representation of $q$, as before we easily get  $$t = \phi_1\;  \text{ and  }\textbf{v} = \left( 0, \cos (\phi_2), \sin (\phi_2) \cos (\phi_3),\ldots, \sin (\phi_2) \sin (\phi_3) \cdots \sin (\phi_n) \right).$$
	Therefore,  with respect to the orthonormal  basis  $\left( e_{2}, e_{3}, \ldots, e_{n+1} \right)$ of $T_{p}\mathbb{H}^{n}$, the geodesic normal coordinates of  $q \in \mathbb{H}^{n}$ centered at $p$  also have the same expression as given in \eqref{Geodesic normal coordinates on sphere}.

	
	\begin{rmk} \label{relation between r and phi_1}
		For $q \in M,$ let $X_{i}(q), 1 \leq i \leq n$ be the geodesic normal coordinates of $q.$ If  $q = \exp_{p}(\textbf{v})$ for some $\textbf{v} \in T_{p}M$, then the geodesic  distance $r_p(q)$ of  $q$ from $p$ is $\norm{v}_p.$  More over,
		$$ \sum_{i=1}^{n} X_{i}^{2}(q) = \langle \textbf{v}, \textbf{v} \rangle_{p}=  \phi_{1}(q).$$  Thus  $r_{p}(q) =  \sum_{i=1}^{n} X_{i}^{2}(q) = \phi_{1}(q)$. Therefore,  we  use  $r$ interchangeably with   $\phi_{1}$.
	\end{rmk}
	
	\begin{rmk}
		We can also define the notion of rotation on  $M$ via exponential map as follows:   for $q \in M$ and a rotation $R$ on $T_p M,$
		\begin{align*}
			R (q) := \exp_{p}\left(R \left(\exp_{p}^{-1}(q)\right)\right).
		\end{align*}
		In particular,   in terms of the geodesic normal coordinates, using \eqref{Geodesic normal coordinates on sphere} we get
		\begin{align} \label{change in coordinates}
			\begin{array}{rcll}
				R_{i,j}^{\frac{2 \pi}{2}} \left(X_{1}, \ldots, X_{i}, \ldots,  X_{j}, \ldots, X_{n}\right) = \left(X_{1}, \ldots, -X_{i}, \ldots,  -X_{j}, \ldots, X_{n}\right) \\
				R_{i,j}^{\frac{2 \pi}{4}} \left(X_{1}, \ldots, X_{i}, \ldots,  X_{j}, \ldots, X_{n}\right) = \left(X_{1}, \ldots, -X_{j}, \ldots,  X_{i}, \ldots, X_{n}\right).
			\end{array}
		\end{align}
	\end{rmk}
	
	Next, we compute the gradient of the geodesic normal coordinates \eqref{Geodesic normal coordinates on sphere}. This result is used in the proof of the main result.
	
	\begin{lem} \label{lem: gradient on sphere and hyperbolic space}
		For  $M=\mathbb{S}^n$ or $M=\mathbb{H}^n,$ let $\{X_i:i=1,2,\ldots, n \}$ be as given in  \eqref{Geodesic normal coordinates on sphere}. Let  $L_M \leq \frac{\pi}{2}$, if $M=\mathbb{S}^n$. Then for any smooth function $ g: [0,L_M) \rightarrow \mathbb{R}$ and for $1 \leq  i <j \leq n$, the followings hold:
		\begin{align} \label{grad:inner product}
			\big \langle \nabla \left( g(r) X_i \right), \nabla \left( g(r) X_j \right)  \big \rangle = \frac{\left(r g'(r) + g(r) \right)^{2}}{r^{2}} X_{i} X_{j} - \frac{g^{2}(r)}{\sin_{M}^{2}(r)}  X_{i} X_{j},
		\end{align}
		\begin{align} \label{grad:norm}
			\bigg | \nabla \left( g(r) \frac{X_i}{r} \right)\bigg |^{2} = \left( g'(r)\right)^{2} \left(\frac{X_i}{r}\right)^{2} + \frac{g^{2}(r)}{\sin_{M}^{2}(r)} \left( 1 - \left(\frac{X_i}{r}\right)^{2} \right).
		\end{align}
		
	\end{lem}
	\begin{proof} First we consider  $M=\mathbb{S}^n$. We view of Remark \ref{relation between r and phi_1}, we replace $r$ with $\phi_1$ and do the calculations.  By a straight forward calculation, we can see that for $i \leq n-1$,
		\begin{align} \label{gradient cal 1} \nonumber
			\nabla \left( g(\phi_1) X_i \right) &= \frac{d}{d \phi_1} \left(\phi_1 g(\phi_1) \right) \frac{X_{i}}{\phi_1} \frac{\partial}{\partial \phi_1} + \sum_{k = 2}^{i} g(\phi_1) \frac{\cos (\phi_k)}{\prod_{l=1}^{k-1} \sin^{2} (\phi_l)} \frac{X_{i}}{\sin (\phi_k)} \frac{\partial}{\partial \phi_k} \\
			& \quad - g(\phi_1) \frac{\sin (\phi_{i+1})}{\prod_{l=1}^{i}\sin^{2} (\phi_l)} \frac{X_{i}}{\cos (\phi_{i+1})} \frac{\partial}{\partial \phi_{i+1}}.
		\end{align}
		For $i = n$,
		\begin{align} \label{gradient cal 2}
			\nabla \left( g(\phi_1) X_n \right) = \frac{d}{d \phi_1} \left(\phi_1 g(\phi_1) \right) \frac{X_{n}}{\phi_1} \frac{\partial}{\partial \phi_1} + \sum_{k = 2}^{n} g(\phi_1) \frac{\cos (\phi_k)}{\prod_{l=1}^{k-1} \sin^{2} (\phi_l)} \frac{X_{n}}{\sin (\phi_k)} \frac{\partial}{\partial \phi_k}.
		\end{align}
		From \eqref{gradient cal 1} and \eqref{gradient cal 2}, we conclude that for $1 \leq i < j \leq n$,
		\begin{align*}
			\langle \nabla \left( g(\phi_1) X_i \right), \nabla \left( g(\phi_1) X_j \right) \rangle & = \left( \frac{d}{d \phi_1} \left(\phi_1 g(\phi_1) \right)\right)^{2} \frac{X_{i} X_{j}}{\phi_{1}^{2}} + \sum_{k = 2}^{i} g^{2}(\phi_1) \frac{\cos^{2} (\phi_k)}{\prod_{l=1}^{k-1} \sin^{2} (\phi_l)} \frac{X_{i} X_{j}}{\sin^{2} (\phi_k)} \\
			& \quad - g^{2}(\phi_1) \frac{X_{i} X_{j}}{\prod_{l=1}^{i}\sin^{2} (\phi_l)}, \\
			& = \left( \frac{d}{d \phi_1} \left(\phi_1 g(\phi_1) \right)\right)^{2} \frac{X_{i} X_{j}}{\phi_{1}^{2}} + \sum_{k = 2}^{i-1} g^{2}(\phi_1) \frac{\cos^{2} (\phi_k)}{\prod_{l=1}^{k-1} \sin^{2} (\phi_l)} \frac{X_{i} X_{j}}{\sin^{2} (\phi_k)} \\
			& \quad - g^{2}(\phi_1) \frac{X_{i} X_{j}}{\prod_{l=1}^{i-1}\sin^{2} (\phi_l)}, \\
			& = \left( \frac{d}{d \phi_1} \left(\phi_1 g(\phi_1) \right)\right)^{2} \frac{X_{i} X_{j}}{\phi_{1}^{2}} - g^{2}(\phi_1) \frac{X_{i} X_{j}}{\sin^{2} (\phi_1)}.
		\end{align*}
		Using similar calculations, we can establish  \eqref{grad:inner product} for $M=\mathbb{H}^n$, and \eqref{grad:norm} for both  $M=\mathbb{S}^n$  and  $M=\mathbb{H}^n.$ 
	\end{proof}
	Next, we prove some orthogonality results of test functions which are crucial for proving the main result.
	\subsection{Orthogonality of test functions}
	The orthogonality of test functions in $L^{2}(\Omega)$ and $H^{1}(\Omega)$ for $\Omega \subset \mathbb{R}^{n}$ has been proved in \cite{ABD}. In the following proposition, we generalize this result for $\Omega \subset \mathbb{S}^{n}$ and $\mathbb{H}^{n}$.
	\begin{prop}\label{Prop:integral inequalities}
		Let $\Omega$ be a bounded domain in $M$ and $\{X_i:i=1,2,\ldots, n\}$ be as given in \eqref{Geodesic normal coordinates on sphere}. Let $g: [0,\infty) \rightarrow \mathbb{R}$ be any smooth function. Then for any $i,j \in \{1, 2, \ldots, n\}$ with $i \neq j$ and $m \in \mathbb{N}\cup\{0\}$, the following assertions hold:
		\begin{enumerate}[label = (\roman*)]
			\item If $\Omega$ is centrally symmetric, then
			\begin{align*}
				\int_{\Omega} g(r) X_{i} X_{j}^{2m} \ dV_{X} = 0 \text{ and } \int_{\Omega} g(r) X_{i}^{2m+1} \ dV_{X} = 0.
			\end{align*}  
			\item If $n \geq 3$ and $\Omega$ is symmetric of order $2$, then
			\begin{align*}
				\int_{\Omega} g(r) X_{i} X_{j}^{m} \ dV_{X} = 0 \text{ and } \int_{\Omega} g(r) X_{i}^{2m+1} \ dV_{X} = 0.   
			\end{align*}
			
			\item If $\Omega$ is symmetric of order $4$, then 
			\begin{align*}
				\int_{\Omega} g(r) X_{i} X_{j} \ dV_{X} = 0.
			\end{align*}
			
			\item If $\Omega$ is symmetric of order $4$, then there exist constants $A_{1}, A_{2}$ such that
			\begin{align*}
				\int_{\Omega} g(r) X_{i}^{2} \ dV_{X} = A_{1} \text{ and } \int_{\Omega} g(r) X_{i}^{4} \ dV_{X} = A_{2}  \text{ for all }  i \in \{1, 2, \ldots, n\}. 
			\end{align*}
			\item If $\Omega$ is symmetric of order $4$, then 
			\begin{align*}
				\int_{\Omega}  \big \langle \nabla \left(g(r) X_{i} \right), \nabla \left(g(r) X_{j} \right) \big \rangle \ dV_{X} = 0.
			\end{align*}
		\end{enumerate}
	\end{prop}
	\begin{proof}
		\begin{enumerate}[label = (\roman*)]
			\item For $X \in T_{p}(M)$, Using the transformation $Y = -X$, we obtain
			\begin{align*}
				\int_{\Omega} g(r) Y_{i} Y_{j}^{2m} \ dV_{Y} = - \int_{\Omega} g(r) X_{i} X_{j}^{2m} \ dV_{X} \text{ and } \int_{\Omega} g(r) Y_{i}^{2m+1} \ dV_{Y} = \int_{\Omega} g(r) X_{i}^{2m+1} \ dV_{X}.   
			\end{align*}
			This implies
			\begin{align*}
				\int_{\Omega} g(r) X_{i} X_{j}^{2m} \ dV_{X} = 0 \text{ and } \int_{\Omega} g(r) X_{i}^{2m+1} \ dV_{X} = 0,   
			\end{align*}
			\noindent which proves our claim.
			
			\item Since $n \geq 3$, choose $k$ such that $k \neq i,j$. Then using the transformation $$Y = R_{i,k}^{\frac{2 \pi}{2}} X \ ((Y_{1}, \ldots, Y_{i}, \ldots, Y_{k}, \ldots, Y_{n}) = (X_{1}, \ldots, - X_{i}, \ldots, - X_{k}, \ldots, X_{n})),$$ we get
			\begin{align*}
				\int_{\Omega} g(r) Y_{i} Y_{j}^{m} \ dV_{Y} = - \int_{\Omega} g(r) X_{i} X_{j}^{m} \ dV_{X} \text{ and } \int_{\Omega} g(r) Y_{i}^{2m+1} \ dV_{Y} = - \int_{\Omega} g(r) X_{i}^{2m+1} \ dV_{X}.  
			\end{align*}
			Hence the desired result follows.
			
			\item The transformation $$Y = R_{i,j}^{\frac{2 \pi}{4}} X \ ((Y_{1}, \ldots, Y_{i}, \ldots, Y_{j}, \ldots, Y_{n}) = (X_{1}, \ldots, - X_{j}, \ldots, X_{i}, \ldots, X_{n}))$$ yields
			\begin{align*}
				\int_{\Omega} g(r) Y_{i} Y_{j} \ dV_{Y} = - \int_{\Omega} g(r) X_{i} X_{j} \ dV_{X}.
			\end{align*}
			This implies the desired expression.
			
			\item For any $i \neq 1$, applying the transformation $$Y = R_{1,i}^{\frac{2 \pi}{4}} X \ ((Y_{1}, \ldots, Y_{i}, \ldots, Y_{n}) = (- X_{i}, \ldots, X_{1}, \ldots, X_{n})),$$ for all  $i \in \{2, \ldots, n\}$, we can write  
			\begin{align*}
				\int_{\Omega} g(r) X_{i}^{2} \ dV_{X} = \int_{\Omega} g(r) Y_{1}^{2} \ dV_{Y}, \\
				\int_{\Omega} g(r) X_{i}^{4} \ dV_{X} = \int_{\Omega} g(r) Y_{1}^{4} \ dV_{Y}. 
			\end{align*}
			This provide the required constants $A_{1} = \int_{\Omega} g(r) Y_{1}^{2} \ dV_{Y}$ and $A_{2} = \int_{\Omega} g(r) Y_{1}^{4} \ dV_{Y}$.
			
			\item This follows easily from Lemma \ref{lem: gradient on sphere and hyperbolic space}.
		\end{enumerate}
	\end{proof}


	\section{Proof of Theorem \ref{main result} and a few remarks} \label{sec:main result}
	
	We consider the Rayleigh quotient,  
	\begin{align*} 
		\mathcal{R}_\Omega(u) = \frac{\int_{\Omega} \vert \nabla u \vert^2 dV}{\int_{\Omega} u ^2 dV},\; u\in H^1(\Omega)\setminus\{0\}.
	\end{align*}
	By the Courant-Fischer minimax formula, the Neumann eigenvalues of \eqref{Neumann problem} are characterized by the following variational formula
	\begin{align} \label{minmaxOmega}
		\mu_{j+1}(\Omega) = \min_{E \in \mathcal{H}_j} \max_{0 \neq u \in E} \mathcal{R}_\Omega(u),
	\end{align}
	where $\mathcal{H}_j$ is the set of all $j$-dimensional subspaces of the Sobolev space $H^1(\Omega)$ that are orthogonal to the subspace of constant functions. In particular, the first positive Neumann eigenvalue $\mu_{2}(\Omega)$ is given by
	\begin{equation} \label{minmax2}
		\mu_{2}(\Omega) = \min_{u \in H^1(\Omega) \setminus \{0\}} \bigg \{  \mathcal{R}_\Omega(u) \bigg| \int_{\Omega} u \ dV = 0 \bigg \}.
	\end{equation}

	\subsection{Proof of \eqref{main result: order 2}}
	For $r>R_1,$ let $G(r)=G_1(r),$ where $G_1$ is defined as in Proposition \ref{prop:comparison with eigenvalue}.
	Now by Proposition \ref{Prop:integral inequalities}, we have
	\begin{align*}
		\int_{\Omega} \frac{G(r)}{r} X_{i} \ dV = 0 \text{ for all } i = 1, 2, \ldots, n,
	\end{align*}
	where $X_{1}, X_{2}, \ldots, X_{n}$ are the geodesic normal coordinates mentioned in \eqref{Geodesic normal coordinates on sphere}. Thus using $\frac{G(r)}{r} X_{i}$ as a test function in \eqref{minmax2} and summing over $i = 1, 2, \ldots, n$, we get
	\begin{align} \label{ineq: variational main result}
		\mu_{2}(\Omega) \sum_{i=1}^{n}  \int_{\Omega} \left( \frac{G(r)}{r} X_{i} \right)^2 dV  \leq \sum_{i=1}^{n} \int_{\Omega} \bigg \vert \nabla \left( \frac{G(r)}{r} X_{i} \right) \bigg \vert^2 dV.
	\end{align}
	Using Lemma \ref{lem: gradient on sphere and hyperbolic space} and the fact $\sum_{i=1}^{n}X_i^2=r^2$,  we obtain
	\begin{align*}
		\sum_{i=1}^{n} \bigg | \nabla \left( \frac{G(r)}{r} X_{i} \right)\bigg |^{2} = \left( G'(r)\right)^{2} + \frac{n-1}{\sin_{M}^{2}(r)} G^{2}(r).
	\end{align*}
	Now from  \eqref{ineq: variational main result} and Proposition \ref{prop:comparison with eigenvalue} we get
	\begin{align*}
		\mu_{2}(\Omega) \leq \frac{\int_{\Omega} \left( \left( G'(r)\right)^{2} + \frac{n-1}{\sin_{M}^{2}(r)} G^{2}(r) \right) dV}{\int_{\Omega} G^{2}(r) dV}\le \mu_{1,1}.
	\end{align*}
	Since $\mu_{1,1} = \mu_2(B_{R_{2}} \setminus \overline{B}_{R_{1}})$, we get the desired result.
	
	\subsection{Proof of \eqref{main result: order 4}} In view of Remark \ref{rmk: multiplicity of annular eigenvalues}, it is enough to show that $\mu_{n+1}(\Omega) \leq \mu_{1,1}$. For this, consider the n-dimensional subspace of $H^1(\Omega)$
	\begin{align*}
		E = \text{ span } \bigg \{ \frac{G(r)}{r} X_{i} \ \Big | \ i = 1,2, \ldots, n \bigg \}.
	\end{align*}
	By Proposition \ref{Prop:integral inequalities}, 
	for   $i,j = 1, 2, \ldots, n$ with $i \neq j$, we  have
	\begin{equation}\label{orthogonality}
		\begin{aligned} 
			\int_{\Omega} \frac{G(r)}{r} X_{i} \ dV & = 0,\\
			\int_{\Omega}   \left(\frac{G(r)}{r} X_{i} \right) \left(\frac{G(r)}{r} X_{j} \right) \, dV  &= 0, \\
			\int_{\Omega} \Big \langle \nabla \left(\frac{G(r)}{r} X_{i} \right),  \nabla \left(\frac{G(r)}{r} X_{j} \right) \Big \rangle \, dV &= 0.
		\end{aligned}
	\end{equation}
	Moreover, there exist constants $C$ and $D$ such that, for all $i = 1,2, \ldots, n$,
	\begin{equation}\label{eqn:main1}
		\begin{aligned}
			&\int_{\Omega} \left( \frac{G(r)}{r} X_i \right)^{2} \, dV =  C,\\
			&\int_{\Omega} \Big | \nabla \left( \frac{G(r)}{r} X_i \right) \Big |^{2} \, dV = \int_{\Omega} \left[ \left( G'(r)\right)^{2} \left(\frac{X_i}{r}\right)^{2}  +  \frac{G^{2}(r)}{\sin_{M}^{2}(r)} \left( 1 - \left(\frac{X_i}{r }\right)^{2} \right)\right] \, dV = D,
		\end{aligned}
	\end{equation}
	where the second equality follows from Lemma \ref{lem: gradient on sphere and hyperbolic space}. Now, by summing up the above equations over $i,$ we get
	\begin{equation}\label{eqn:main2}
		\begin{aligned}
			C&=\frac{1}{n}\int_{\Omega} G^{2}(r) \, dV, \qquad 
			D=\frac{1}{n}\int_{\Omega} \left( \left( G'(r)\right)^{2} + \frac{n-1}{\sin_{M}^{2}(r)} G^{2}(r) \right) \, dV.
		\end{aligned}
	\end{equation}
	Let $u\in E\setminus \{0\}$. Then there exists $c=(c_1,c_2,\ldots, c_n)\in \mathbb{R}^n\setminus \{0\}$ such that 
	$u = \sum_{i=1}^{n} c_{i} \frac{G(r)}{r} X_{i}.$ Now, using \eqref{orthogonality} and  \eqref{eqn:main1} we obtain:
	\begin{align*} 
		\mathcal{R}_\Omega(u)=\frac{\int_{\Omega} \vert \nabla u \vert^2 \, dV}{\int_{\Omega} u^2 \, dV}
		& = \frac{ \int_{\Omega}  \vert \sum_{i=1}^{n} c_{i} \nabla \left( \frac{G(r)}{r} X_{i} \right) \Big \vert^2 dV }{  \int_{\Omega} \left(  \sum_{i=1}^{n} c_{i} \frac{G(r)}{r} X_{i} \right) ^2 dV }
		= \frac{ \sum_{i=1}^{n} c_{i}^{2} \int_{\Omega} \Big \vert \nabla \left( \frac{G(r)}{r} X_{i} \right) \Big \vert^2 dV }{ \sum_{i=1}^{n} c_{i}^{2}  \int_{\Omega}  \left( \frac{G(r)}{r} X_{i} \right) ^2 dV } 
		= \frac{D}{C}. 
	\end{align*}
	Therefore, from \eqref{eqn:main2} and Proposition \ref{prop:comparison with eigenvalue}, we conclude that
	$$\mathcal{R}_\Omega(u)=\frac{\int_{\Omega} \left( \left( G'(r)\right)^{2} + \frac{n-1}{\sin_{M}^{2}(r)} G^{2}(r) \right) \, dV}{\int_{\Omega} G^{2}(r) \, dV}\le \mu_{1,1}=\mu_2(B_{R_{2}} \setminus \overline{B}_{R_{1}}), u\in E\setminus \{0\}.$$
	Thus  \eqref{minmaxOmega} yields:
	$$\mu_{n+1}(\Omega)\le \max_{u\in E\setminus \{0\}}\mathcal{R}_\Omega(u)\le \mu_2(B_{R_{2}} \setminus \overline{B}_{R_{1}}).$$
	\begin{rmk}
		Our proof also work for $M = \mathbb{R}^{n}$ by considering standard cartesian coordinates as normal coordinates centered at the origin with Riemannian metric $g = dr^{2} + r^{2} g_{\mathbb{S}^{n-1}}$.
	\end{rmk}
	\begin{rmk}
		It is natural to anticipate  the main Theorem \ref{main result} for symmetric domains in smooth Riemannian manifold $M = [0,R) \times \mathbb{S}^{n-1}$ equipped with the warped product metric $g = dr^{2} + h^{2}(r) g_{\mathbb{S}^{n-1}}$. Here function $h \in C^{\infty}([0,R))$, $h(r) > 0$ for $r \in (0, R)$, $h'(0) = 1$ and $h^{(2k)}(0)=0$ for all integers $k \geq 0$.     For such manifolds 
		\begin{align*}
			\text{ tr(A(r))} = \frac{1}{h^{n-1}(r)} \frac{d}{dr} \left(h^{n-1}(r) \right) \text{ and } \lambda_{1}(S(r)) = \frac{n-1}{h^{2}(r)}
		\end{align*}
		Thus,  for  proving  Lemma \ref{eigenvalue comparison}, we need   $(\text{tr}(A(r)))' = - \lambda_{1}(S(r))$. This is true if and only if 
		\begin{align*}
			h(r)h''(r) - h'^{2}(r) = - 1,
		\end{align*}
		and the only solution of this ODE are  $h(r) = r, \sin r$ and $\sinh r$.  Thus our proof works only for $h(r) = r, \sin r$ and $\sinh r$. 
	\end{rmk}

	\begin{rmk}
		Next we list some related open problems.
		
		\begin{enumerate}

			\item For $M=\mathbb{S}^n$, we proved the main result only for then domains contained in the hemisphere. Is it possible to extend   the main results for domains that are not contained in the hemisphere?
			
			\item Rank-$1$ symmetric spaces naturally generalise the space forms as the isometry group acts transitively on the unit tangent bundle. Thus establishing   Szeg\"{o}-Weinberger inequality for the higher Neumann eigenvalues for the  bounded domains in rank-1 symmetric spaces and also in manifolds with bounded curvature seems to be some interesting problems. The challenging part, in these cases, is to find the appropriate geodesic normal coordinates.
			
		\end{enumerate}
	\end{rmk}

	\textbf{Acknowledgement:} Authors would like to thank Prof. G. Santhanam for useful discussions and his suggestions.
	\normalem
	
\end{document}